\documentclass{article}

\usepackage[english]{babel}

\usepackage[letterpaper,top=2cm,bottom=2cm,left=3cm,right=3cm,marginparwidth=1.75cm]{geometry}
\usepackage{amsmath}
\usepackage{amssymb}
\usepackage{amsthm}
\usepackage{graphicx}
\usepackage{bigints}
\usepackage[colorlinks=true, allcolors=blue]{hyperref}
\usepackage{graphicx}
\usepackage{mathtools}
\usepackage{tikz}
\usepackage{tikz-cd}
\usetikzlibrary{cd}
\usepackage{setspace}
\usepackage{tikz-network}
\usetikzlibrary{positioning} 
\usetikzlibrary{automata}

\newcommand{\defeq}{\vcentcolon=}

\newtheorem{manualtheoreminner}{Theorem}
\newenvironment{manualtheorem}[1]{%
  \renewcommand\themanualtheoreminner{#1}%
  \manualtheoreminner
}{\endmanualtheoreminner}

\newtheorem{theorem}{Theorem}[section]

\newtheorem{lemma}{Lemma}[section]
\newtheorem{definition}{Definition}[section]

\newtheorem{example}{Example}[section]
\newtheorem*{remark}{Remark}
\date{}

\newcommand{\M}[2]{M_{#1,#2}}
\newcommand{\Mo}[2]{\overline{M}_{#1,#2}}

\newcommand{\PI}[1]{\pi_{[#1]}^\ast(\Delta)}
\newcommand{\I}[2]{\int_{\overline{M}_{#1,#2}}}
\title{A combinatorial proof of the $\lambda_g$ conjecture in genus 2}
\author{Taylor Rogers, Renzo Cavalieri}
\begin{document}

\maketitle

\begin{abstract}
    We give a simple combinatorial proof of the $\lambda_g$ conjectue in genus 2.  We use a description of the class $\lambda_2$ as a linear combination of boundary strata, and show the conjecture follows inductively from applications of the projection formula, string equation, and dilaton equation.
\end{abstract}


\section{Introduction}
One way that we study the geometry of moduli spaces of curves is by computing how its subvarieties intersect.  When a collection of subvarieties intersect in a finite number of points, such a number is called an intersection number, and it  is generally difficult to calculate.  One way to approach this is to exploit the recursive structure that intersection numbers possess.  For monomials of $\psi$-classes, examples of recursive structure are given by the string and dilaton equations (see Section \ref{sssec:psi} or \cite{VakilGW}).  We are interested in computing intersection numbers of monomials of $\psi$-classes againist the pullback  of the stratum
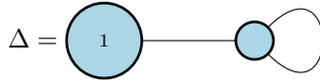
\[ \Delta =  \parbox{4cm}{
\begin{tikzpicture}
\Vertex[x=0, label = 1, size=1]{A}
\Vertex[x=2, size = 0.5]{B}
\draw (B) to [out =315, in=45, looseness = 10] (B);
\draw (A) -- (B);
\end{tikzpicture}}
\]

via the forgetful morphism $\pi_{[n]}:\Mo{2}{n} \to \Mo{2}{}$. To do so, we apply the projection formula, string equation, and dilaton equation and obtain the following theorem.
\begin{manualtheorem}{3.1}
For any positive integer $n$ and    partition $k_1+\dots+k_n=n+1$ of the integer $n+1$,  
\[\int_{\PI{n}}\psi_1^{k_1}\dots\psi_n^{k_n}=\frac{1}{24}\binom{n+1}{k_1,k_2,\dots,k_n}\]
\end{manualtheorem}
  Theorem \ref{thm:pullback} allows us to give a simple proof of the $\lambda_g$ conjecture for genus two,
  \begin{manualtheorem}{3.2}
For any positive integer $n$ and partition $k_1+\dots+k_n=n+1$ of the integer $n+1$,  
\[\I{2}{n}\lambda_2\psi_1^{k_1}\dots\psi_n^{k_n}=\frac{7}{24\cdot 8\cdot 30}\binom{n+1}{k_1,k_2,\dots,k_n}\]
\end{manualtheorem}
\noindent The proof of this theorem follows from Theorem \ref{thm:pullback} and a formula of Pixton \cite{pixton} which allows us to write $\lambda$-classes as linear combinations of strata. \\ 
\indent In Section 2 we establish the relevant background information on the Deligne-Mumford compactification of the moduli space of pointed curves and tautological intersection theory within this moduli spaces.  Then in Section 3 we establish the results above, in particular: the relations analogous to the dilaton and string equations and the proofs of Theorem \ref{thm:pullback} and Theorem \ref{thm:lambdatwo}.


\section{Background}
\subsection{Moduli Space of Stable n-Pointed Curves}
Here we give a concise presentation of the theory on moduli spaces of curves  which is needed for our result. More information and references regarding the material introduced in this section may be found in \cite{kockQC} and \cite{VakilGW}.

By $M_{0,n}$  we denote the fine moduli space whose points parameterize $n$-tuples of distinct points in $\mathbb{P}^1$ up to projective equivalence.  To understand $M_{0,n}$ we first look at the case for $n=4$.  Two quadruples $P\defeq(p_1,p_2,p_3,p_4)$ and $Q\defeq(q_1,q_2,q_3,q_4)$ are \textbf{projectively equivalent} if there exists M\"{o}bius transformation $\varphi$ such that $\varphi(p_i)=q_i$ for $i = 1, \ldots, 4$.  For any three distinct points $x,y,z\in\mathbb{P}^1$, there exists a unique M\"{o}bius transformation $\varphi\in\textrm{Aut}(\mathbb{P}^1)$ such that $x$, $y$, and  $z$ are taken to  $0$, $1$ and $\infty$ respectively under $\varphi$. It follows that any quadruple $P$ is equivalent  to a unique quadruple of the form $(0,1,\infty, t)$, and therefore $M_{0,4}$ is  isomorphic to $\mathbb{P}^1\setminus\{0,1,\infty\}$.  Generalizing this idea, we find $M_{0,n}$ as the fine moduli space given by the $(n-3)$-fold product of $M_{0,4}$ without diagonals, i.e. 
\[M_{0,n}\cong \M_{0,4}\times\dots\times M_{0,4}\setminus\{\cup \delta_{ij}\},\]
where $\delta_{ij}$ denotes the locus of points where the $i$-th and $j$-th coordinates are equal.

Since $M_{0,4}\cong \mathbb{P}^1\setminus\{0,1,\infty\}$ has dimension one,  this implies that 
$\dim{M_{0,n}}=n-3$.
The space $\M{0}{n}$ is not compact, however, it admits a compactification of interest: the Deligne-Mumford compactification, $\Mo{0}{n}$.  This moduli space parameterizes a broader class of curves which we now characterize. 
\begin{definition}
A \textbf{tree of projective lines} is a connected curve such that 
\begin{enumerate}
    \item Each irreducible component is isomorphic to a projective line.
    \item  The points of intersection of the components are ordinary double points, i.e., nodes.
    \item The fundamental group is trivial.
\end{enumerate}
We call the irreducible components \textbf{twigs}.
\end{definition}
\begin{definition}
For $n\geq 3$ a \textbf{stable n-pointed rational curve} is a tree of projective lines with $n$ marked distinct points in the smooth locus, such that the sum of the number of marked points and nodes on each twig is at least three.
\end{definition}
$\Mo{0}{n}$ parameterizes isomorphism classes of stable $n$-pointed curves and contains $\M{0}{n}$ as a dense open set.  We call $\Mo{0}{n}\setminus \M{0}{n}$ the boundary of $\Mo{0}{n}$ and points in the boundary correspond to  nodal marked curves.  The space $\M{0}{n}$ admits a natural stratification, given by the equivalence class for the equivalence relation that declares two points equivalent  if their corresponding curves are homeomorphic as pointed curves.  
\begin{definition}
Given a rational stable $n$-pointed curve $C$ with marked points $p_1,\dots, p_n$ its \textbf{dual graph} is a tree defined to have:
\begin{enumerate}
    \item a vertex for each twig of $C$.
    \item an edge for each node of $C$ connecting the appropriate vertices.
    \item a labeled half edge corresponding to each marked point on the appropriate vertex.
\end{enumerate}
\end{definition}

Two marked curved are homeomorphic if and only if they have the same dual graph, and therefore dual graphs index the strata of $\M{0}{n}$.

\begin{example} \normalfont
Below are three examples of dual graphs of stable rational pointed curves\footnote{Legs on these graphs should be labeled, but, throughout the paper, we omit unimportant labels to keep the pictures cleaner.}, 
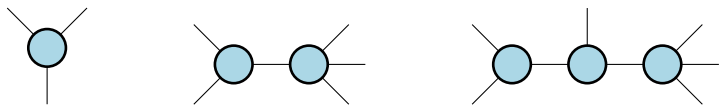
\[
\begin{tikzpicture}
\Vertex[x=0, size = 0.5]{B}
\node[above left =0.5cm of B] (D) {};
\node[above right =0.5cm of B] (E) {};
\node[below = 0.5cm of B] (F) {};
\draw (B) -- (D);
\draw (B) -- (E);
\draw (B) -- (F);
\end{tikzpicture}
\ \ \ \ \ \ \ \
\begin{tikzpicture}
\Vertex[x=0, size = 0.5]{A}
\Vertex [x=1, size = 0.5]{B}
\node[above right =0.5cm of B] (D) {};
\node[right =0.5cm of B] (E) {};
\node[below right= 0.5cm of B] (F) {};
\draw (B) -- (D);
\draw (B) -- (E);
\draw (B) -- (F);
\draw (A) -- (B);
\node[above left = 0.5cm of A] (G) {};
\node[below left = 0.5cm of A] (H) {};
\draw (A) -- (G);
\draw (A) -- (H);
\end{tikzpicture}
\ \ \ \ \ \ \ \
\begin{tikzpicture}
\Vertex[x=0, size = 0.5]{A}
\Vertex [x=2, size = 0.5]{B}
\node[above right =0.5cm of B] (D) {};
\node[right =0.5cm of B] (E) {};
\node[below right= 0.5cm of B] (F) {};
\draw (B) -- (D);
\draw (B) -- (E);
\draw (B) -- (F);
\node[above left = 0.5cm of A] (G) {};
\node[below left = 0.5cm of A] (H) {};
\draw (A) -- (G);
\draw (A) -- (H);
\Vertex[x=1, size = 0.5]{C}
\draw (A) -- (C);
\draw (C) -- (B);
\node[above = 0.5cm of C] (I) {};
\draw (C) --(I);
\end{tikzpicture}\]
On the other hand, the next two are the dual graphs of an unstable pointed rational curve and a stable but not rational curve respectively.
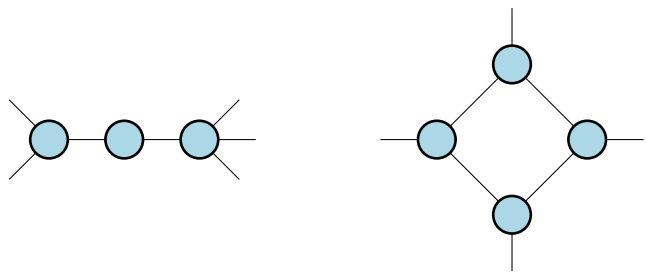
\[
\parbox{4cm}{
\begin{tikzpicture}
\Vertex[x=0, size = 0.5]{A}
\Vertex [x=2, size = 0.5]{B}
\node[above right =0.5cm of B] (D) {};
\node[right =0.5cm of B] (E) {};
\node[below right= 0.5cm of B] (F) {};
\draw (B) -- (D);
\draw (B) -- (E);
\draw (B) -- (F);
\node[above left = 0.5cm of A] (G) {};
\node[below left = 0.5cm of A] (H) {};
\draw (A) -- (G);
\draw (A) -- (H);
\Vertex[x=1, size = 0.5]{C}
\draw (A) -- (C);
\draw (C) -- (B);
\end{tikzpicture}}
\ \ \ \ \ \ \ \
\parbox{4cm}{
\begin{tikzpicture}
\Vertex[x=0, y=1, size = 0.5]{A}
\Vertex[x=0, y=-1, size=0.5]{C}
\Vertex[x=-1, y=0, size = 0.5]{B}
\Vertex[x=1, y=0, size=0.5]{D}
\draw (A) -- (B);
\draw (B) -- (C);
\draw (C) -- (D);
\draw (D) -- (A);
\node[above = 0.5cm of A] (E) {};
\node[left = 0.5cm of B] (F) {};
\node[below = 0.5cm of C] (G) {};
\node[right = 0.5cm of D] (H) {};
\draw (A) -- (E);
\draw (B) -- (F);
\draw (C) -- (G);
\draw (D) -- (H);
\end{tikzpicture}}
\]
\end{example}
This compactified moduli space has coarse generalization for curves of any genus; the space parameterizing curves of genus $g$ with $n$ marked points is denoted $\Mo{g}{n}$.  The theory of the moduli space of curves of genus $g$ is more complex, however the dimension of this space faithfully generalizes that of the genus $0$  space with $\dim \Mo{g}{n}=3g-3+n$ \cite{harrisMod}.  For our purposes we will only need the cases for $g\leq 2$ and really only care about $g=0$ and $g=1$.  
\subsection{Natural Morphisms}
We now discuss two natural morphisms on these moduli spaces: the forgetful and gluing morphisms.  We also briefly cover a form of the projection formula which will be of use.
\subsubsection{Forgetful Morphism}
Given an $(n+1)$-pointed curve $(C,p_1,\dots,p_{n+1})$ one can forget $p_{n+1}$ to obtain an $n$-pointed curve $(C,p_1,\dots,p_n)$, this idea leads to the forgetful morphism $\pi_{n+1}:\Mo{g}{n+1}\rightarrow \Mo{g}{n}$. \\
\indent Of course, the full definition requires addressing subtleties, we will address them now.  First, no forgetful morphism exists for $(g,n)=(0,3)$ or $(g,n)=(1,1)$, due to stability conditions.  Secondly, we can't always just forget a marked point, since doing so may result in an unstable curve.  In such cases,  an extra stabilization process must be introduced, called contraction.  This happens in the following two cases.
\begin{enumerate}
    \item If $p_{n+1}$ is on a twig (name we reserve for rational components) with two nodes and no other marked points, then we contract that twig after forgetting $p_{n+1}$.  For example 
    
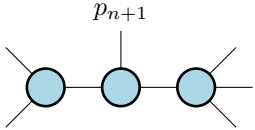
\[
\begin{tikzpicture}
\Vertex[x=0, size = 0.5]{A}
\Vertex [x=2, size = 0.5]{B}
\node[above right =0.5cm of B] (D) {};
\node[right =0.5cm of B] (E) {};
\node[below right= 0.5cm of B] (F) {};
\draw (B) -- (D);
\draw (B) -- (E);
\draw (B) -- (F);
\node[above left = 0.5cm of A] (G) {};
\node[below left = 0.5cm of A] (H) {};
\draw (A) -- (G);
\draw (A) -- (H);
\Vertex[x=1, size = 0.5]{C}
\draw (A) -- (C);
\draw (C) -- (B);
\node[above = 0.5cm of C] (I) {$p_{n+1}$};
\draw (C) --(I);
\end{tikzpicture}\]
   will be sent to 
   
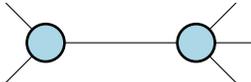
\[
\begin{tikzpicture}
\Vertex[x=0, size = 0.5]{A}
\Vertex [x=2, size = 0.5]{B}
\node[above right =0.5cm of B] (D) {};
\node[right =0.5cm of B] (E) {};
\node[below right= 0.5cm of B] (F) {};
\draw (B) -- (D);
\draw (B) -- (E);
\draw (B) -- (F);
\node[above left = 0.5cm of A] (G) {};
\node[below left = 0.5cm of A] (H) {};
\draw (A) -- (G);
\draw (A) -- (H);
\draw (A) -- (B);
\end{tikzpicture}\]
   under the forgetful morphism $\pi_{n+1}$.
   \item If $p_{n+1}$ is on a twig with exactly one other marked point ${p_i}$ and exactly one node, then the twig is contracted after forgetting $p_{n+1}$ and ${p_i}$ is placed on what used to be the node.  For example, 
   
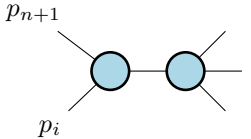
\[\begin{tikzpicture}
   \Vertex[x=0, size = 0.5]{A}
   \Vertex[x=1, size = 0.5]{B}
   \node[above left =0.5cm of A](C){$p_{n+1}$};
   \node[below left =0.5cm of A](D){$p_i$};
    \node[above right =0.5cm of B](E){};
   \node[below right =0.5cm of B](F){};
   \node[right=0.5cm of B](G){};
   \draw (A)--(C);
\draw (A)--(D);
    \draw (B)--(E);
   \draw (B)--(F);
   \draw (B)--(G);
   \draw(A)--(B);
   \end{tikzpicture}\]
   is sent to 
   
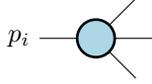
\[\begin{tikzpicture}
   \Vertex[x=0, size=0.5]{B}
   \node[left =0.5cm of B](D){$p_i$};
    \node[above right =0.5cm of B](E){};
   \node[below right =0.5cm of B](F){};
   \node[right=0.5cm of B](G){};
   \draw (B)--(D);
    \draw (B)--(E);
   \draw (B)--(F);
   \draw (B)--(G);
   \end{tikzpicture}\]
   under the forgetful morphism $\pi_{n+1}$.
\end{enumerate}


\subsubsection{Gluing Morphisms}
Another natural idea is to identify two marked points, thus gluing two curves together.  That is, given an $(n+1)$-pointed curve  of genus $g$ and and $(n'+1)$-pointed curve of genus $g'$ we may obtain a genus $g+g'$ curve with $(n+n')$-marked points, defining a map $\textrm{gl}:\Mo{g}{n+1}\times \Mo{g'}{n'+1}\rightarrow \Mo{g+g'}{n+n'}$.

 One could do this with a genus $g$ curve with $(n+2)$-marked points to obtain a curve of genus $g+1$ with $n$-marked points, giving rise to a gluing morphism $\textrm{gl}:\Mo{g}{n+2}\rightarrow \Mo{g+1}{n}$. 
 
 In fact one may naturally generalize this idea to gluing multiple pairs of points, and observe that the closure of any stratum in $\Mo{g}{n}$ is the image of some appropriately defined gluing morphism.

\begin{example}\normalfont
\label{ex:glue334}
Given the following two curves 
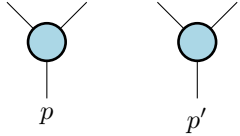
\[
\begin{tikzpicture}
\Vertex[x=0, size = 0.5]{B}
\node[above left =0.5cm of B] (D) {};
\node[above right =0.5cm of B] (E) {};
\node[below = 0.5cm of B] (F) {$p$};
\draw (B) -- (D);
\draw (B) -- (E);
\draw (B) -- (F);
\Vertex[x=2, size = 0.5]{A}
\node[above left =0.5cm of A] (H) {};
\node[above right =0.5cm of A] (I) {};
\node[below = 0.5cm of A] (J) {$p'$};
\draw (A) -- (H);
\draw (A) -- (I);
\draw (A) -- (J);
\end{tikzpicture}\]
we may identify $p$ and $p'$ under the gluing morphism $\textrm{gl}:\Mo{0}{3}\times \Mo{0}{3}\rightarrow \Mo{0}{4}$ to obtain the curve 
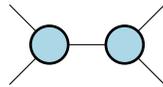
\[
\begin{tikzpicture}
\Vertex[x=1,size=0.5]{A}
\Vertex[x=0,size=0.5]{B}
\draw(A)--(B);
\node[above left =0.5cm of B] (D) {};
\node[below left =0.5cm of B] (E) {};
\node[above right =0.5cm of A] (F) {};
\node[below right =0.5cm of A] (G) {};
\draw(B)--(D);
\draw(B)--(E);
\draw(A)--(F);
\draw(A)--(G);
\end{tikzpicture}\]
\end{example}
\begin{example}\normalfont
Given the following curve, 
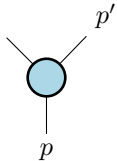
\[
\begin{tikzpicture}
\Vertex[x=0, size = 0.5]{B}
\node[above left =0.5cm of B] (D) {};
\node[above right =0.5cm of B] (E) {$p'$};
\node[below = 0.5cm of B] (F) {$p$};
\draw (B) -- (D);
\draw (B) -- (E);
\draw (B) -- (F);
\end{tikzpicture}
\]
we may identify $p'$ and $p$ via the gluing morphism $\textrm{gl}:\Mo{0}{3}\rightarrow \Mo{1}{1}$ to obtain 
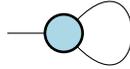
\[
\begin{tikzpicture}
\Vertex[x=2, size = 0.5]{B}
\draw (B) to [out =315, in=45, looseness = 10] (B);
\node[left=0.5cm of B] (A){};
\draw(A)--(B);
\end{tikzpicture}\]
\end{example}

\subsubsection{The Projection Formula}
The projection formula allows us to intersect, via appropriate applications of pushforwards and pullbacks, Chow classes that live on two different spaces connected by a well behaved morphism. For a more in depth coverage of the projection formula consult \cite{andallthat}, Section 1.3.6.

Given a flat and proper morphism $f:X\rightarrow Y$, a Chow class $\beta$ of $Y$ and $\alpha$ of $X$, the following relation holds:
\begin{description}
\item[Projection Formula] \[f_*((f^*\beta)\cdot \alpha)=(f_*\alpha)\cdot \beta \]
\end{description}
with multiplication in the Chow ring.  
The following picture illustrates the content of the projection formula.

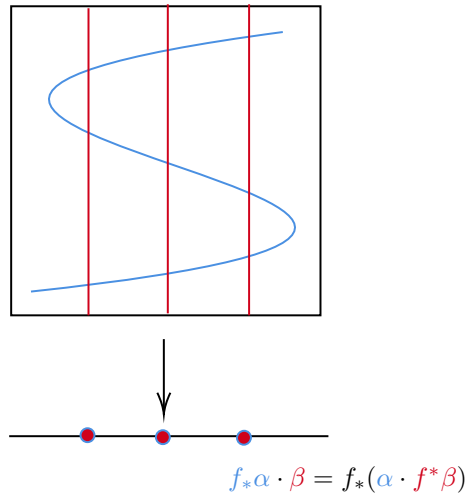
\tikzset{every picture/.style={line width=0.75pt}} 

\begin{tikzpicture}[x=0.75pt,y=0.75pt,yscale=-1,xscale=1]

\draw   (242,22) -- (398,22) -- (398,178) -- (242,178) -- cycle ;
\draw    (241,239) -- (402,239) ;
\draw [color={rgb, 255:red, 74; green, 144; blue, 226 }  ,draw opacity=1 ]   (252,166) .. controls (641,124) and (17,83) .. (379,35) ;
\draw  [color={rgb, 255:red, 74; green, 144; blue, 226 }  ,draw opacity=1 ][fill={rgb, 255:red, 208; green, 2; blue, 27 }  ,fill opacity=1 ][line width=0.75]  (355.91,239.9) .. controls (355.9,237.96) and (357.47,236.34) .. (359.4,236.27) .. controls (361.34,236.21) and (362.91,237.73) .. (362.92,239.66) .. controls (362.92,241.6) and (361.36,243.22) .. (359.43,243.29) .. controls (357.49,243.35) and (355.92,241.83) .. (355.91,239.9) -- cycle ;
\draw  [color={rgb, 255:red, 74; green, 144; blue, 226 }  ,draw opacity=1 ][fill={rgb, 255:red, 208; green, 2; blue, 27 }  ,fill opacity=1 ][line width=0.75]  (315,239.5) .. controls (315,237.57) and (316.57,236) .. (318.5,236) .. controls (320.43,236) and (322,237.57) .. (322,239.5) .. controls (322,241.43) and (320.43,243) .. (318.5,243) .. controls (316.57,243) and (315,241.43) .. (315,239.5) -- cycle ;
\draw  [color={rgb, 255:red, 74; green, 144; blue, 226 }  ,draw opacity=1 ][fill={rgb, 255:red, 208; green, 2; blue, 27 }  ,fill opacity=1 ][line width=0.75]  (277,238.5) .. controls (277,236.57) and (278.57,235) .. (280.5,235) .. controls (282.43,235) and (284,236.57) .. (284,238.5) .. controls (284,240.43) and (282.43,242) .. (280.5,242) .. controls (278.57,242) and (277,240.43) .. (277,238.5) -- cycle ;
\draw    (319,190) -- (319,226) ;
\draw [shift={(319,228)}, rotate = 270] [color={rgb, 255:red, 0; green, 0; blue, 0 }  ][line width=0.75]    (10.93,-3.29) .. controls (6.95,-1.4) and (3.31,-0.3) .. (0,0) .. controls (3.31,0.3) and (6.95,1.4) .. (10.93,3.29)   ;
\draw [color={rgb, 255:red, 208; green, 2; blue, 27 }  ,draw opacity=1 ]   (281,23) -- (281,178) ;
\draw [color={rgb, 255:red, 208; green, 2; blue, 27 }  ,draw opacity=1 ]   (321,21) -- (321,177) ;
\draw [color={rgb, 255:red, 208; green, 2; blue, 27 }  ,draw opacity=1 ]   (362,21) -- (362,178) ;

\draw (350,252.4) node [anchor=north west][inner sep=0.75pt]    {$\textcolor[rgb]{0.29,0.56,0.89}{f_\ast\alpha} \cdot \textcolor[rgb]{0.82,0.01,0.11}{\beta }=f_{*}(\textcolor[rgb]{0.29,0.56,0.89}{\alpha} \cdot  \textcolor[rgb]{0.82,0.01,0.11}{f^\ast\beta })$};

\end{tikzpicture}
In this picture we see the pushforward of $f^*\beta\cdot\alpha$ gives $f_*\alpha\cdot\beta$ which we have represented by taking the points of $\beta$ and outlining them in blue.\\
\indent We remark that forgetful morphisms are flat and proper; hence one can use the projection formula.


\subsection{Tautological Intersection theory on $\Mo{g}{n}$}
We are mainly interested in the intersection numbers of particular Chow classes in $\Mo{g}{n}$.  In these compactified moduli spaces, the strata have a closure whose fundamental class gives a natural Chow class to work with.  In genus 0 and 1, these classes are of particular importance for this project.  In genus 0, these classes generate the Chow ring and in genus 1 they generate a particularly important subring of the Chow ring, the tautological ring \cite{VakilGW}.
\begin{remark} \normalfont
We follow the following convention of \cite{pixton}.  Every dual graph identifies a stratum and the stratum is an image of a gluing morphism.  However, these gluing morphisms are not always $1:1$, so the  fundamental class of the stratum is not always the same as the pushforward of the fundamental class via the gluing morphism. By  dual graphs we denote the pushforward of fundamental classes of products of moduli spaces under the gluing morphism.  As such if a stratum is identified by a dual graph $\Gamma$, the class of the closure of the stratum is then $\frac{1}{|\textrm{Aut}(\Gamma)|}[\Gamma]$.
\end{remark} 
The following examples illustrate this point.
\begin{example}\normalfont
\label{ex:bij}
In Example \ref{ex:glue334}, we obtained a stratum of $\Mo{0}{4}$ by gluing two strata from smaller moduli spaces. The gluing morphism provides a bijection 
\[\textrm{gl}_*:\Mo{0}{3}\times\Mo{0}{3}\rightarrow \Mo{0}{4}\]
Therefore, $ \textrm{gl}_*(1_{\Mo{0}{3} \times \Mo{0}{3}})$ agrees with the class of the stratum represented by the following dual graph
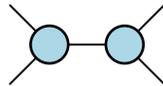
\[
\begin{tikzpicture}
\Vertex[x=1,size=0.5]{A}
\Vertex[x=0,size=0.5]{B}
\draw(A)--(B);
\node[above left =0.5cm of B] (D) {};
\node[below left =0.5cm of B] (E) {};
\node[above right =0.5cm of A] (F) {};
\node[below right =0.5cm of A] (G) {};
\draw(B)--(D);
\draw(B)--(E);
\draw(A)--(F);
\draw(A)--(G);
\end{tikzpicture}\]
 
\end{example}
\begin{example}\normalfont
Though the gluing morphism in Example \ref{ex:bij} is a bijection, injectivity often fails.  Take $\textrm{gl}: \Mo{0}{1,2,*,\bullet}\rightarrow\Mo{1}{2}$, this is a $2:1$ map onto its image.  Take a point in the stratum corresponding to the dual graph, $\Gamma$, depicted below
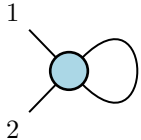
\[\begin{tikzpicture}
\Vertex[x=0, size=0.5]{A}
\node[above left =0.5cm of A] (B) {1};
\node[below left=0.5cm of A] (C) {2};
\draw (A) to [out =315, in=45, looseness = 10] (A);
\draw (A) -- (B);
\draw (A) -- (C);
 \end{tikzpicture}\]
 As noted in the remark, dual graphs denote the pushforward of fundamental classes of products and quotients of moduli spaces under the gluing morphism, here we have:
 \[[\Gamma]=\textrm{gl}_*[1_{\Mo{0}{1,2,*,\bullet}}]\]
 This points' preimage under $\textrm{gl}$ is given by two points which we may schematically depict as
 
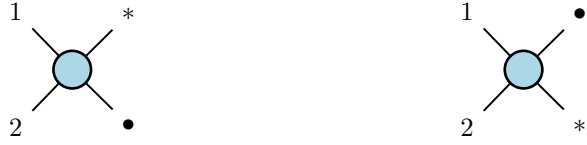
\[\begin{tikzpicture}
\Vertex[x=0, size=0.5]{A}
\node[above left =0.5cm of A] (B) {1};
\node[below left=0.5cm of A] (C) {2};
\node[above right =0.5cm of A] (D) {$*$};
\node[below right=0.5cm of A] (E) {$\bullet$};
\draw (A) -- (B);
\draw (A) -- (C);
\draw (A) -- (D);
\draw (A) -- (E);
\Vertex[x=6, size=0.5]{J}
\node[above left =0.5cm of J] (F) {1};
\node[below left=0.5cm of J] (G) {2};
\node[above right =0.5cm of J] (H) {$\bullet$};
\node[below right=0.5cm of J] (I) {$*$};
\draw (J) -- (F);
\draw (J) -- (G);
\draw (J) -- (H);
\draw (J) -- (I);
 \end{tikzpicture}\]
So, the class of the closure of the stratum is given by $\frac{1}{2}[\Gamma]$.\end{example}
We particularly care about two other families of classes given as Chern classes of some natural bundles on $\Mo{g}{n}$, which we now discuss.


\subsubsection{$\psi$-classes} \label{sssec:psi}
Given $\Mo{g}{n}$ we have the universal curve $\pi: C_{g,n}\rightarrow \Mo{g}{n}$.  Let $s_i$ be the section of $\pi$ corresponding to the $i$-th marked point.  Then the pullback by $s_i$ of the relative dualizing sheaf defines a bundle denoted $\mathbb{L}_i$. We then define the $i$-th \textbf{$\psi$-class} to be the first Chern class of $\mathbb{L}_i$, i.e., \[\psi_i=c_1(\mathbb{L}_i)\in A^{1}(\Mo{g}{n})\]
A more complete treatment of this may be found in $\cite{kocknotes}$.
In $g=0$ and $g=1$, the intersection numbers of monomials of $\psi$-classes are completely determined by the following two recursive structures and initial conditions.
\begin{description}
\item[String Equation] Given $g,n, k_1, \dots, k_n\in \mathbb{Z}^+$ with $k_1+\dots+k_n=3g-3+n+1$ and $2g-2+n>0$ then 
\[\I{g}{n+1}\psi_1^{k_1}\dots\psi_n^{k_n}=\sum_{i=1}^n\I{g}{n}\psi_1^{k_1}\dots \psi_i^{k_i-1}\dots \psi_n^{k_n},\]
adopting the notational convention that $\psi_i^{-1} = 0$.
\item[Dilaton Equation] Given $g,n, k_1, \dots, k_n\in \mathbb{Z}^+$ with $k_1+\dots+k_n=3g-3+n$ and $2g-2+n>0$ then 
\[\I{g}{n+1}\psi_1^{k_1}\dots \psi_n^{k_n}\psi_{n+1}=(2g-2+n)\I{g}{n}\psi_1^{k_1}\dots \psi_n^{k_n} . \]
\item[Initial Conditions] \[\I{1}{1}\psi_1=\frac{1}{24}\qquad \textrm{and}\qquad  \I{0}{3}1=1.\]
\end{description}
The above two relations are called the \textbf{dilaton equation} and the \textbf{string equation} \cite{kocknotes} respectively and they are the most important results of the background section for the purposes of this paper.  For a derivation of the initial condition see \cite{VakilGW} Section 3.13.

\begin{remark} \normalfont
While it is most common to see the string and dilaton equations expressed as identities among intersection numbers, in the course of their proof one readily sees that they are really identities among push-forwards of cycles (\cite{kocknotes}, Section 1.4). In particular, the dilaton equation is derived from the fact that $\pi_{n*}(\psi_n) = (2g-2+n)[1_{\Mo{g}{n}}]$, while string follows from the relation:
\begin{equation}
\label{eq:cyclestring}
    \pi_{n*}(\psi_1^{k_1}\dots\psi_n^{k_n})=\sum_{j=1}^n \psi_1^{k_1}\dots \psi_i^{k_i-1}\dots \psi_n^{k_n}.
\end{equation}
\end{remark}


\subsubsection{$\lambda$-classes}
Another type of class is obtained from the push-forward of the relative dualizing sheaf of the universal curve, which we call the \textbf{Hodge bundle}, denoted $\mathbb{E}_{g,n}$.  With this bundle  we may define \textbf{$\lambda$-classes} in a similar fashion to the $\psi$-classes as
 \[\lambda_i=c_i(\mathbb{E}_{g,n}).\] 
An important property of these classes is that they are stable under pullback via forgetful morphisms, i.e. $\pi_n^\ast \lambda_i = \lambda_i$.

The $\lambda_g$ conjecture (\cite{lambda},  later proven in \cite{lambdaPf}) gives a simple formula for calculating intersection numbers of monomials of $\psi$-classes on $\Mo{g}{n}$ along with a factor of the $\lambda_g$ class.  For ease of reference, the theorem is explicitly stated: 
\begin{theorem}[$\lambda_g$ Conjecture, \cite{lambdaPf}] \label{lambdag}
Let $k_1,\dots,k_n\in\mathbb{Z}^+$ such that $k_1+\dots+k_n=2g-3+n$.  Then 
\[\I{g}{n}\psi_1^{k_1}\dots\psi_n^{k_n}\lambda_g=\binom{2g+n-3}{k_1,\dots,k_n}\I{g}{1}\psi_1^{2g-2}\lambda_g\]
with initial condition
\[\I{g}{1}\psi_1^{2g-2}\lambda_g=\frac{2^{2g-1}-1}{2^{2g-1}}\frac{|B_{2g}|}{(2g)!}\]
where $B_n$ is the $n$-th Bernoulli number. \textnormal{\cite{Wiki}}
\end{theorem}
We are interested in giving an elementary proof of this theorem for $g=2$.  In doing so we will make use of a formula of Pixton, found in section 6 of \cite{pixton}, which gives $\lambda_g$ as a linear combination of strata.  For genus two, this formula has the form 
\begin{equation} \label{for:pix}
    \lambda_2 = \frac{1}{240} 
    \left[\parbox{1cm}{\scalebox{0.5}{\begin{tikzpicture}[scale = 1]
\Vertex[x=0, label = 1, size=1]{A}
\Edge[loopsize=1cm](A)(A)
\node[above right =0.3cm of A] (D) {$\psi$};
\end{tikzpicture}}}
\right] 
+\frac{1}{1152}
\left[\parbox{1.2cm}{\scalebox{0.5}{\begin{tikzpicture}[scale = 1]
\Vertex[x=4, size =0.5]{B}
\draw (B) to [out =315, in=45, looseness = 10] (B);
\draw (B) to [out =225, in=135, looseness = 10] (B);
\end{tikzpicture}}}
\right],
\end{equation}
where the decoration of $\psi$ on a half edge of a graph means pushing forward the $\psi$ class at the corresponding mark via the gluing morphism.





\section{Results}

Throughout the following section, ``nontrivial partition" refers to any integer partition of a number $n$ that is not $1+1+\dots+1=n$.  Further, $\Delta$  refers to the dual graph
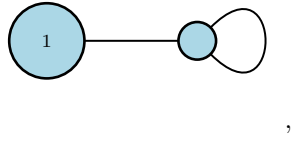
\[
\begin{tikzpicture}
\Vertex[x=0, label = 1, size=1]{A}
\Vertex[x=2, size = 0.5]{B}
\draw (B) to [out =315, in=45, looseness = 10] (B);
\draw (A) -- (B);
\end{tikzpicture},
\]
and the correponding class in the tautological ring of $\Mo{2}{}$.

 Recall that $\dim \Mo{g}{n}=3g-3+n$, hence the intersection number of a monomial of $\psi$-classes over $\Mo{g}{n}$ is nonzero if and only if the monomial's degree is $3g-3+n$.  Indeed, for each $n$, there are $2^n$ strata in $\PI{n}$, since each stratum in $\PI{n-1}$ gives two strata in $\PI{n}$.  Calculating the intersection numbers from Theorem \ref{lambdag} amounts to calculating a sum of $2^n$ intersection numbers, one for each stratum of $\PI{n}$.  If we let $i$ denote one such stratum, $\int_{i}\Psi^K$ is equal to a product of two intersection numbers by Fubini's theorem, one being the intersection number of the monomial of $\psi$-classes over the genus 1 twig of $\Delta$ and the other being the intersection number of the monomial of $\psi$-classes over the genus 0 twig.


\begin{example} \normalfont \label{example:31}
To make the above discussion more concrete, let's investigate the $n=1$ case.  For $n=1$, $\PI{1}$ contains two strata   
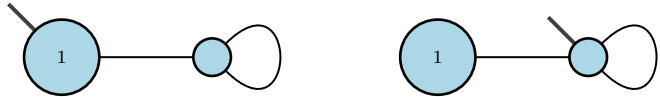
\[
\begin{tikzpicture}
\Vertex[x=0, label = 1, size=1]{A}
\Vertex[x=2, size = 0.5]{B}
\draw (B) to [out =315, in=45, looseness = 10] (B);
\draw (A) -- (B);
\node[above left =0.5cm of A] (D) {};
\Edge (D)(A)
\Vertex[x=5, label = 1, size=1]{F}
\Vertex[x=7, size = 0.5]{G}
\draw (G) to [out =315, in=45, looseness = 10] (G);
\draw (F) -- (G);
\node[above left =0.5cm of G] (H) {};
\Edge (H)(G)
\end{tikzpicture}
\]
which we refer to as $A$ and $B$ respectively.  Now $\Delta$ is of codimension 2.  Therefore the pullback of $\Delta$ under $\pi$, $\PI{1}$, is of codimension 2 in the four dimensional $\M{2}{1}$,  and is therefore dimension 2.  \\
\indent Since $\dim \PI{1}=2$, a monomial of $\psi$-classes must be of degree 2 in order for the intersection number to be nonzero; further, $n=1$ so only one $\psi$-class is available to us.  Hence, the intersection number of interest is $\int_{\PI{1}}\psi^2$.\\
\indent The pullback of $\Delta$ under $\pi$ consists of A and B so 
\[\int_{\PI{1}}\psi^2=\int_A\psi^2+\int_B\psi^2\]
Then by Fubini's theorem and Section \ref{sssec:psi} we have 
\begin{align*}
    \int_{\PI{1}}\psi^2=\int_A\psi^2+\int_B\psi^2 &= \I{1}{2}\psi^2\I{0}{3}1+\I{1}{1}1\I{0}{4}\psi^2\\
    &=\frac{1}{24}
\end{align*}
\end{example}
The only partitions of $n+1$ that can appear as exponent vectors for the monomials of $\psi$-classes are nontrivial.  Indeed, for any positive integer $n$, intersection numbers are nonzero if and only if the degree of the monomial of $\psi$-classes in question is of degree $n+1$.  However, we only have one $\psi$-class for each marked point, i.e. $n$ $\psi$-classes.  Therefore the partition $1+1+\dots+1=n+1$ is unattainable. \\
It's convenient for us to adopt the following notational convention:
\begin{itemize}
    \item Given a nontrivial partition $K=k_1+\dots+k_n$ of $3g-3+n\in \mathbb{Z}^+$, 
    \[\I{g}{n}\Psi^K\defeq \I{g}{n}\psi_1^{k_1}\dots\psi_n^{k_n}\]
\end{itemize}
Let $c:X\rightarrow \textrm{pt}$ be the constant morphism and we observe that the integral symbol is just a notation for the degree of the push forward of a class via $c$:  $\int_X\alpha = c_*(\alpha)$.  We have the following commutative diagram, 
\begin{equation}\label{pushdia}
    \begin{tikzcd}
	{\overline{M}_{g,n+1}} \\
	& {\textrm{pt}} \\
	{\overline{M}_{g,n}} \\
	{\overline{M}_g}
	\arrow["{\pi_{[n-1]}}", from=3-1, to=4-1]
	\arrow["{\pi_n }"', from=1-1, to=3-1]
	\arrow["{\pi_{[n]}}"', bend right, from=1-1, to=4-1]
	\arrow["c", from=1-1, to=2-2]
	\arrow["{\tilde{c}}"', from=3-1, to=2-2]
\end{tikzcd}
\end{equation}
which will be useful in the proofs of the following lemmas. 

We are now ready to provide the relations analogous to the dilaton and string equations with the following two lemmas. 

\begin{lemma}\label{lem:dil}
Let $K$ be a nontrivial partition of $n\in \mathbb{Z}^+$, then 
\[\int_{\PI{n}}\Psi^K\psi_n=(n+1)\int_{\PI{n-1}}\Psi^K\]
\end{lemma}

\begin{proof}
For a nontrivial partition $K$ of $n$, we start by expressing each $\psi$ class as the pull-back via the morphism forgetting the last mark, plus the appropriate boundary correction (see \cite{kocknotes}, Lemma 1.3.1)
\[\Psi^K\psi_n=\prod_i(\pi^*_n\psi_i+D_{i,n})^{k_i}\psi_n\]
where $D_{i,n}$ denotes the boundary divisor where the $i$-th and $n$-th marked points are together on a rational component with no other mark.  Notice for all $i$, $\psi_n\cdot D_{i,n}=0$ and therefore
\begin{equation}\label{eq:puba}
\Psi^K\psi_n=\pi_n^*\left(\prod_i\psi_i^{k_i}\right)\psi_n\end{equation}
Chasing the commutative diagram \eqref{pushdia}, one has 
\begin{align*}
 \int_{\PI{n}}\Psi^K\psi_n&=c_*(\pi^*_n(\PI{n-1})\Psi^K\psi_n\\
 &=\tilde{c}_*\pi_{n*}(\pi^*_n(\PI{n-1})\Psi^K\psi_n).
\end{align*}
then by \eqref{eq:puba} we may bring $\Psi^K$ within the parenthesis, yielding 
\[\tilde{c}_*\pi_{n*}(\pi^*_n(\PI{n-1}\Psi^K)\psi_n).\]
By the projection formula this reduces to 
\[\tilde{c}_*(\PI{n-1}\Psi^K\cdot \pi_{n*}\psi_n).\]
This allows us to apply the dilaton equation, finally giving 
\begin{align*}
    \tilde{c}_*(\PI{n-1}\Psi^K\cdot\pi_{n*}\psi_n)&=(2g-2+n)\int_{\PI{n-1}}\Psi^K\\
    &=(n+1)\int_{\PI{n-1}}\Psi^K.
\end{align*}
\end{proof}


\begin{lemma} \label{lem:string}
Let $K$ be a nontrivial  partition of $n+1$, $n\in \mathbb{Z}^+$ with $k_n=0$ (in other words the length of $K$ is at most $n-1$). If $j\in [n-1]$ then we define $K(j)$ as
\[K(j)\defeq \begin{cases} 
      (k_1,\dots,k_j-1,\dots,k_n) & k_j>0 \\
      0 & k_j=0
   \end{cases}
\]
Then
\[\int_{\PI{n}}\Psi^K=\sum_{j\in [n-1]}\int_{\PI{n-1}}\Psi^{K(j)}\]
\end{lemma}
\begin{proof}

We start again with diagram \eqref{pushdia}, 
\begin{align*}
    \int_{\PI{n}}\Psi^K  &= c_*(\PI{n}\Psi^K)\\
    &=\tilde{c}_*\pi_{n*}(\pi_n^*(\PI{n-1})\Psi^K).
\end{align*}
Then, by the projection formula
\[\tilde{c}_*\pi_{n*}(\pi_n^*(\PI{n-1})\Psi^K)=\tilde{c}_*(\PI{n-1}\cdot \pi_{n*}\Psi^K).\]
Finally the cycle theoretic string equation \eqref{eq:cyclestring} gives 
\[\tilde{c}_*(\PI{n-1}\cdot \pi_{n*}\Psi^K)=\tilde{c}_*\left(\PI{n-1}\cdot \sum_{j\in[n-1]}\Psi^K\right).\]
In terms of intersection numbers 
\[\int_{\PI{n}}\Psi^K=\sum_{j\in [n-1]}\int_{\PI{n-1}}\Psi^{K(j)}.\]
\end{proof}


This next lemma is just a generalization of Pascal's rule for binomial coefficients. It may be proved by an elementary induction argument, so we omit its proof.  
\begin{lemma}\label{lem:pas}
For $m\in \mathbb{Z}^+$, $m \geq 2$, $k_1,\dots,k_m\in \mathbb{Z}^+$ with $n=k_1+\dots+k_m\geq 1$ we have 
\[\sum_{j\in[n-1]}\binom{n-1}{K(j)}=\binom{n}{K}\]
\end{lemma}


These three lemmas allow us to prove the following first major result of this paper.
\begin{theorem} \label{thm:pullback}
Given a nontrivial partition $K$ of $n+1$ with $n\in \mathbb{Z}^+$, 
\begin{equation}\label{eq:bino}\int_{\PI{n}}\Psi^K=\frac{1}{24}\binom{n+1}{K}\end{equation}
\end{theorem}
\begin{proof}
We proceed by induction with Example \ref{example:31} providing the base case.
Assume \eqref{eq:bino} holds for all nontrivial partitions of $n$ and let $K$ be a partition of $n+1$ with one component equal to $0$.  Then we may apply Lemma \ref{lem:string} giving 
\[\int_{\PI{n}}\Psi^K=\sum_{j\in[n-1]}\int_{\PI{n-1}}\Psi^{K(j)}\]
By the inductive hypothesis 
\[ \sum_{j\in[n-1]}\int_{\PI{n-1}}\Psi^{K(j)}=\sum_{j\in[n-1]}\frac{1}{24}\binom{n}{K(j)} \]
Them Lemma \ref{lem:pas} and the initial condition yield
\[ \sum_{j\in[n-1]}\frac{1}{24}\binom{n}{K(j)}=\frac{1}{24}\binom{n+1}{K} \]
If $K$ is a nontrivial partition of $n+1$ with no component equal to $0$, then some component must be $1$.  Indeed, if all components were greater than $1$ then their sum must be greater than $n+1$.  Hence we may write $\Psi^K=\Psi^{K'}\psi_n$ and apply Lemma \ref{lem:dil} finding
\[ \int_{\PI{n}}\Psi^{K'}\psi_n=(n+1)\int_{\PI{n-1}}\Psi^{K'}. \]
So by the inductive hypothesis 
\[ (n+1)\int_{\PI{n-1}}\Psi^{K'}=(n+1)\frac{1}{24}\binom{n}{K'}=\frac{1}{24}\binom{n+1}{K} \]
\end{proof}


\noindent With Theorem \ref{thm:pullback} we now use Pixton's formula  \eqref{for:pix} to give a proof of the $\lambda_g$ conjecture for $g=2$. 
\begin{theorem}\label{thm:lambdatwo}
Given a nontrivial partition $K$ of $n+1$ for $n\in \mathbb{Z}^+$, 
\[\I{2}{n}\lambda_2\Psi^K=\frac{7}{24\cdot 8\cdot 30}\binom{n+1}{K}.\]
\end{theorem}

\begin{proof}
We first convert  \eqref{for:pix}  into a linear combination of strata.  We have $    \left[\parbox{.7cm}{\scalebox{0.3}{\begin{tikzpicture}[scale = 1]
\Vertex[x=0, label = 1, size=1]{A}
\Edge[loopsize=1cm](A)(A)
\node[above right =0.3cm of A] (D) {$\psi$};
\end{tikzpicture}}}
\right] =\textrm{gl}_*(\psi_1)$ for the gluing morphism $\textrm{gl}:\Mo{1}{2}\rightarrow\overline{M}_2$ and $\psi_1$ on $\Mo{1}{2}$.

On $\Mo{1}{1}$ we can express the class $\psi_1$ as:  
$$\psi_1  = \frac{1}{24}  \left[\parbox{ 1cm}{\scalebox{0.5}{\begin{tikzpicture}
\Vertex[x=0, size=0.5]{A}
\node[right =0.5 of A](B){};
\draw (A) -- (B);
\draw (A) to [out =225, in=135, looseness = 10] (A);
\end{tikzpicture}}}
\right]$$
Therefore, by \cite[Lemma 1.3.1]{kocknotes} $\psi_1$ on $\Mo{1}{2}$ can be expressed as
\[
\frac{1}{24}
\left[\parbox{ 1cm}{\scalebox{0.5}{\begin{tikzpicture}
\Vertex[x=0, size=0.5]{A}
\node[above right =0.5cm of A] (D) {};
\node[below right =0.5cm of A](P){};
\draw (A) -- (P);
\draw (D)--(A);
\draw (A) to [in =135, out=225, looseness = 10] (A);
\end{tikzpicture}
}}
\right]
+
\left[\parbox{1.5cm}{\scalebox{0.5}{\begin{tikzpicture}
\Vertex[x=5, label = 1, size=1]{F}
\Vertex[x=7, size = 0.5]{G}
\draw (F) -- (G);
\node[above right =0.5cm of G] (H) {};
\node[below right =0.5cm of G] (I) {};
\draw (G) -- (H);
\draw (G) -- (I);
\end{tikzpicture}
}}
\right].
\]

After pushforward we can then rewrite \eqref{for:pix} as:
\begin{equation} \label{eq:almostthere}
\lambda_2=\left(\frac{1}{1152}+\frac{1}{240\cdot24}\right)\Delta_0+\frac{1}{240}\Delta,\end{equation}
where $\Delta_0$ denotes $\textrm{gl}_*(\overline{M}_{0,5})$ via the gluing morphism that glues two distinct pairs of nodes.

From \eqref{eq:almostthere} and the fact that $\lambda_2$ is stable under pull-back, we have 
\begin{align*}
     \I{2}{n}\Psi^{K}\lambda_2 &= \I{2}{n}\Psi^{K}\pi_{[n]}^\ast\left(\left(\frac{1}{1152}+\frac{1}{240\cdot 24}\right)\Delta_0+\frac{1}{240}\Delta\right) \\
    &= \left(\frac{1}{1152}+\frac{1}{240\cdot 24}\right)\int_{\pi^*_{[n]}(\Delta_0)}\Psi^K+\frac{1}{240}\int_{\PI{n}}\Psi^K
\end{align*}
Since $\pi_{[n]}^\ast(\Delta_0)=\textrm{gl}_*(\overline{M}_{0,2g+n})$ and in our case $g=2$, we have 
\begin{equation} \label{last}
\int_{\pi^*_{[n]}(\Delta_0)}\Psi^K= \int_{\overline{M}_{0,2g+n}}\Psi^K=\binom{2g+n-3}{K}=\binom{n+1}{K}\end{equation}
Therefore by \eqref{last} and Theorem \ref{thm:pullback}
\begin{align*}
    \left(\frac{1}{1152}+\frac{1}{240\cdot 24}\right)\int_{\pi^*_{[n]}(\Delta_0)}\Psi^K+\frac{1}{240}\int_{\PI{n}}\Psi^K &= \left(\frac{1}{1152}+\frac{1}{240\cdot 24}+\frac{1}{240\cdot 24}\right)\binom{n+1}{K}\\
    &=\frac{7}{8\cdot 24\cdot 30}\binom{n+1}{K}.
\end{align*}
\end{proof}

\section{Acknowledgments}
The first author would like to acknowledge Professor Renzo Cavalieri for suggesting the topic of this paper and his invaluable mentorship.

\bibliographystyle{alpha}
\bibliography{biblio.bib}

\begin{thebibliography}{MPS21}

\bibitem[EH16]{andallthat}
David Eisenbud and Joe Harris.
\newblock {\em 3264 and all that: A second course in algebraic geometry}.
\newblock Cambridge University Press, 2016.

\bibitem[FP03]{lambdaPf}
C.~Faber and R.~Pandharipande.
\newblock Hodge integrals, partition matrices, and the $\lambda_g$ conjecture.
\newblock {\em Annals of Mathematics}, 157(1):97--124, 2003.

\bibitem[GP98]{lambda}
Ezra Getzler and Rahul Pandharipande.
\newblock Virasoro constraints and the chern classes of the hodge bundle.
\newblock {\em Nuclear Physics B}, 530(3):701--714, 1998.

\bibitem[HM06]{harrisMod}
Joe Harris and Ian Morrison.
\newblock {\em Moduli of curves}, volume 187.
\newblock Springer Science \& Business Media, 2006.

\bibitem[Koc]{kocknotes}
Joachim Kock.
\newblock Notes on psi classes.

\bibitem[KV07]{kockQC}
Joachim Kock and Israel Vainsencher.
\newblock {\em An invitation to quantum cohomology: Kontsevich's formula for
  rational plane curves}, volume 249.
\newblock Springer Science \& Business Media, 2007.

\bibitem[MPS21]{pixton}
Samouil Molcho, Rahul Pandharipande, and Johannes Schmitt.
\newblock The hodge bundle, the universal 0-section, and the log chow ring of
  the moduli space of curves.
\newblock {\em arXiv preprint arXiv:2101.08824}, 2021.

\bibitem[Vak08]{VakilGW}
R.~Vakil.
\newblock {\em The Moduli Space of Curves and Gromov--Witten Theory}, pages
  143--198.
\newblock Springer Berlin Heidelberg, Berlin, Heidelberg, 2008.

\bibitem[Wik20]{Wiki}
Lambda g conjecture.
\newblock \url{https://en.wikipedia.org/wiki/Lambda_g_conjecture}, Nov 2020.

\end{thebibliography}

\end{document}